\documentclass[11pt]{amsart}
\usepackage{amssymb}
\usepackage{amsthm}
\usepackage{amstext}
\usepackage{verbatim}
\usepackage[usenames,dvipsnames]{color}

\newtheorem{thm}{Theorem}[section]
\newtheorem{lem}[thm]{Lemma}
\newtheorem{cor}[thm]{Corollary}

\theoremstyle{definition}
\newtheorem{df}[thm]{Definition}

\theoremstyle{remark}
\newtheorem{rem}[thm]{Remark}
\newtheorem{rems}[thm]{Remarks}

\newcommand{\mg}{\marginpar}

\newcommand{\sq}{\subseteq}

\newcommand{\vS}{\varSigma}
\newcommand{\vO}{\varOmega}

\newcommand{\vT}{\varTheta}

\newcommand{\restr}{\upharpoonright}
\newcommand{\s}{\sigma}

\newcommand{\mf}{\mathfrak}
\newcommand{\vY}{\varUpsilon}
\newcommand{\E}{\mathbb{E}}
\newcommand{\T}{\mathbb{T}}
\newcommand{\vX}{\varXi}
\newcommand{\rp}{(\cnp,\csp)}
\newcommand{\cnp}{N}
\newcommand{\agp}{S}
\newcommand{\csp}{X}
\newcommand{\cip}{W}

\newcommand{\cnpw}{counting process }
\newcommand{\cnpB}{counting process}

\newcommand{\agpw}{aggregate process }
\newcommand{\agpB}{aggregate process}

\newcommand{\cspB}{size process}

\newcommand{\sagp}{S}

\def\N{{\mathbb N}}
\def\R{{\mathbb R}}

\newcommand{\auth}[1]{\textcolor{black}{\texttt{#1}}}

\newcommand{\artp}[1]{\textrm{#1}}
\newcommand{\artn}[1]{\textbf{#1}}
\newcommand{\artj}[1]{\emph{#1}}
\newcommand{\artb}[1]{\emph{#1}}

\newcommand{\Ddown}{$\bf{D\Downarrow D\Downarrow D\Downarrow D\Downarrow
D\Downarrow D\Downarrow D\Downarrow D\Downarrow D\Downarrow
D\Downarrow }$}
\newcommand{\Dup}{$\bf {D\Uparrow D\Uparrow D\Uparrow D\Uparrow D\Uparrow
D\Uparrow D\Uparrow D\Uparrow D\Uparrow D\Uparrow }$}

\newcommand{\arch}[1]{\textcolor[rgb]{0.,0,0.}{#1}}

\newcommand{\adc}[1]{\textcolor{blue}{#1}}

\title{SOME MARTINGALE CHARACTERIZATIONS OF COMPOUND MIXED POISSON PROCESSES}
\author{D.P. Lyberopoulos and N.D. Macheras}
\begin{document}
\date{\today}


\begin{abstract}
Some martingale characterizations of compound mixed Poisson processes are proven, extending S. Watanabe's \cite{wa} martingale characterization of Poisson processes as well as the main result of \cite{lm1}, concerning martingale characterizations of mixed Poisson processes.

\smallskip


\end{abstract}
\thanks{D.P.L. would like to aknowledge that a part of this
work was conducted at the {\sl Department of Statistics and Insurance Science in the University of Piraeus}.  
D.P.L. is also indebted to the Public Benefit Foundation 
{\sl Alexander S. Onassis}, which supported this research, under the Programme 
of Scholarships for Hellenes.}

\maketitle


\section*{Introduction}\label{i}
Mixed Poisson 
processes (MPPs for short) play an important role in many branches of applied probability, for instance in insurance mathematics and point process theory (cf. \cite{gr} for more information). In particular, structural properties of MPPs have always been of specific importance in the field on insurance mathematics, since they are widely in modeling on counting processes, especially in non-life insurance (see Albrecht \cite{alb} for a general survey). 

S.Watanabe \cite{wa} 
provided in 1964 a martingale characterization of Poisson processes 
within the class of counting processes with continuous compensators, as those with deterministic compensators. O. Lundberg \cite{lu} 
gave in 1940 a martingale characterization of MPPs with finite mean value within birth processes. D. Pfeifer \cite{pfe} 
and D. Pfeifer and U. Heller \cite{pfeh}, 
proved in 1987 a variation of Lundberg's martingale characterization of MPPs within birth processes.     
B. Grigelionis \cite{gri}, Theorem 1, extended in 1998 Lundberg's martingale characterization of MPPs with arbitrary mixing distribution $U$ 
(MPP($U$) for short) within general counting processes. 

Note that the definition of a MPP involving the notion of a birth process is equivalent to that of a MPP($U$) (see \cite{lmt1}, Proposition 3.1), while each MPP with structural parameter 
an almost surely positive random variable $\vT$ (MPP($\vT$) for short) on a probability space $(\vO,\vS,P)$ is a MPP($U$) (see \cite{lmt1}, Theorem 3.1) the inverse direction is not in general true, since it is not always possible, given a MPP($U$), to construct a $P$-almost surely non-negative random variable $\vT$ such that $P_{\vT}=U$ (see \cite{fr3}, 343M). On the other hand, assuming that there exists such 
a random variable $\vT$, it is not in general possible to construct a regular conditional probability of $P$ over $U$ consistent with 
$\vT$, since for non-perfect measures $P$ on $\vS$ it is impossible to do it (see \cite{fa}, Theorem 4).  

To the best of our knowledge, the first result on martingale characterization of MPPs with a structural parameter $\vT$ has been given in \cite{lm1}, Theorem 4.10 along with \cite{lm1err}. In this paper we investigate some martingale characterizations for compound mixed Poisson processes (CMPPs for short). 

In Section \ref{d:cmcppdis} we recall the necessary definitions of compound mixed Poisson processes (CMPPs for short) and regular conditional probabilities (r.c.p. for short) as well as some preparatory results, concerning the reduction of a CMPP under a probability measure $P$ to ordinary compound Poisson processes under the probability measures of the corresponding regular conditional probability, proven in \cite{lm3}.

In Section \ref{m:chm} we first provide a characterization of a CMPP in terms of the martingale property of a certain transformation of 
the \agpw $\agp$ (see Theorem \ref{17}) and then we characterize a CMPP in terms of a certain transform of the \cnpw $\cnp$ (see Theorem \ref{27}). 
The latter, which is the main result of this paper, yields among others Watanabe's martingale characterization of Poisson processes, 
and reduces to the main result of \cite{lm1}, that is Theorem 4.10.


\section{Preliminaries}\label{pr}
$\N$ and $\R$ stand for the natural and the real numbers, respectively, while $\R_+:=\{x\in\R:x\geq0\}$. If $d\in\N$, then $\R^d$ denotes the Euclidean space of dimension $d$.

Given a probability space $(\vO,\vS,P)$, a set $N\in\vS$ with $P(N)=0$ is called a $P$-{\bf null set} (or a null set for simplicity). 
For random variables $X,Y:\vO\longrightarrow\R$ we write $X=Y\;$ $P$-almost surely ($P$-a.s. for short), if 
$P(X\neq Y)=0$.

If $A\subseteq\vO$, then $A^c:=\vO\setminus A$, while $\chi_A$ denotes the indicator (or characteristic) function of the set $A$. 
For a map $f:D\longrightarrow{\R}$ and for a non-empty set $A\subseteq{D}$ we denote by $f\upharpoonright{A}$ the restriction of $f$ to $A$. 
The identity map from $\vO$ onto itself is denoted by $id_{\vO}$. 
The $\sigma$-algebra generated by a family $\mathcal{G}$ of subsets of $\vO$ 
is denoted by $\sigma(\mathcal{G})$.

For any Hausdorff topology $\mf{T}$ on $\vO$, by ${\mf B}(\vO)$ we denote the {\bf Borel $\sigma$-algebra} on $\vO$, i.e. the $\sigma$-algebra generated by $\mf{T}$, while $\mf{B}:=\mf{B}(\R)$ stands for the Borel $\sigma$-algebra of subsets of $\R$. 
By $\mathcal{L}^{\ell}(P)$ we denote the space of all $\vS$-measurable real-valued functions $f$ on $\vO$ such that $\int|f|^{\ell}dP<\infty$ 
for $\ell\in\{1,2\}$.

Functions that are $P$-$\mbox{a.s.}$ equal are not identified. 
We write $\E_P[X\mid{\mathcal{G}}]$ for a version of a conditional expectation (under $P$) 
of $X\in\mathcal{L}^{1}(P)$ given a $\sigma$-subalgebra ${\mathcal{G}}$ of $\vS$.

Given two probability spaces $(\vO,\vS,P)$ and $(\vY,T,Q)$ as well as a $\vS$-$T$-measurable map $X:\vO\longrightarrow\vY$ we 
write $\sigma(X):=\{X^{-1}(B): B\in T\}$ for the $\sigma$-algebra generated by $X$, while 
$\s(\{X_i\}_{i\in I}):=\s\bigl(\bigcup_{i\in{I}}\sigma(X_i)\bigr)$ stands for the $\s$-algebra generated by a family $\{X_i\}_{i\in I}$ of $\vS$-$T$-measurable maps from $\vO$ into $\vY$.
For any given $\vS$-$T$-measurable map $X$ from $\vO$ into $\vY$ the measure $P_X: T\longrightarrow\R$ is the image measure of $P$ under $X$. 
By $\mathbf{K}(\theta)$ we denote an arbitrary probability distribution on $\mf{B}$ with parameter $\theta\in\vX$. In particular, 
$\mathbf{P}(\theta)$ and $\mathbf{Exp}(\theta)$, where $\theta$ is a positive parameter, stand for the law of Poisson and exponential 
distribution, respectively (cf. e.g. \cite{sch}).

Given two real-valued random variables $X,\vT$ on $\vO$,  
a {\bf conditional distribution of $X$ over $\vT$} is a $\sigma(\vT)$-$\mathfrak{B}$-Markov kernel (see \cite{ba}, Definition 36.1 for the definition)
denoted by $P_{X\mid\vT}:=P_{X\mid\sigma(\vT)}$ and satisfying for each $B\in\mf{B}$ the equality
$P_{X\mid\vT}(\bullet,B)=P(X^{-1}(B)\mid\sigma(\vT))(\bullet)$ ${P}\restr\sigma(\vT)$-a.s.. Clearly, for every $\mathfrak{B}_d$-$\mathfrak{B}$-Markov kernel $k$, the map $K(\vT)$ from $\vO\times\mathfrak{B}$ into $[0,1]$ defined by means of
$$
K(\vT)(\omega,B):=(k(\bullet,B)\circ\vT)(\omega)
\quad\mbox{for any}\;\;(\omega,B)\in \vO\times\mathfrak{B}
$$
is a $\sigma(\vT)$-$\mathfrak{B}$-Markov kernel. Then for $\theta=\vT(\omega)$ with $\omega\in\vO$ the probability measures $k(\theta,\bullet)$ are distributions on $\mathfrak{B}$ and so we may write $\mathbf{K}(\theta)(\bullet)$ instead of $k(\theta,\bullet)$. Consequently, in this case $K(\vT)$ will be denoted by 
$\mathbf{K}(\vT)$.

For any real-valued random variables $X$, $Y$ on $\vO$ we say that $P_{X\mid\vT}$ and $P_{Y\mid\vT}$ are $P\restr\sigma(\vT)$-equivalent and we write 
$P_{X\mid\vT}=P_{Y\mid\vT}$ $P\restr\sigma(\vT)$-a.s., if there exists a $P$-null set $N\in\sigma(\vT)$ such that for any $\omega\notin N$ and 
$B\in\mf{B}$ the equality $P_{X\mid\vT}(B,\omega)=P_{Y\mid\vT}(B,\omega)$ holds true.

A family $\{{X}_i\}_{i\in I}$ of random variables 
is {\bf $P$-conditionally identically distributed} over a 
random variable $\vT$, if $P(F\cap{X}_i^{-1}(B))=P(F\cap{X}_j^{-1}(B))$ whenever $i,j\in I$, $F\in\sigma(\vT)$
and $B\in\mf{B}$. 
Furthermore, we say that 
$\{{X}_i\}_{i\in I}$ is {\bf $P$-conditionally (stochastically) independent  given}  
$\vT$, if it is conditionally independent given the $\sigma$-algebra $\sigma(\vT)$; 
for the definition of conditional independence see  e.g. \cite{ct}, page 220.

{\em From now on let $(\vO,\vS,P)$ be an arbitrary but fixed probability space. Unless it is stated otherwise, $\vT$ is a random variable on $\vO$ such that $P_{\vT}\bigl((0,\infty)\bigr)=1$, and we simply write 
``conditionally'' in the place of ``conditionally given $\vT$'' whenever conditioning refers to $\vT$}.


\section{Regular conditional probabilities and a preparatory result}\label{d:cmcppdis}

Let $\cnp:=\{N_t\}_{t\in\R_+}$ be a 
{\bf $P$- \cnpw} with 
exceptional $P$- null set $\vO_N$ (cf. e.g. \cite{sch}, page 17 for the definition). Without loss of generality we may and do assume that 
$\vO_N=\emptyset$. Denote by $\{T_n\}_{n\in\N_0}$ and $\cip:=\{\cip_n\}_{n\in\N}$ the {\bf (claim) arrival process} and 
{\bf (claim) interarrival process}, respectively, associated with $\cnp$ (cf. e.g. \cite{sch}, page 6 for the definitions).
Also let $\csp:=\{X_n\}_{n\in\N}$ be the {\bf (claim) size process} with all $X_n$ positive, and $\agp:=\{S_t\}_{t\in\R_+}$ the {\bf aggregate  process} induced by the counting process $\cnp$ and the size process $\csp$ (cf. e.g. \cite{sch}, page 103 for the definitions). 
For the definition of a risk process $\rp$ on $(\vO,\vS,P)$ we refer to \cite{sch}, page 127.

The \cnpw $\cnp$ is said to be a {\bf mixed Poisson process} on $(\vO,\vS,P)$ with parameter $\vT$ (or a $P$-MPP($\vT$) for short),  if it has conditionally stationary independent increments, such that
\[
P_{N_t\mid\vT}=\mathbf{P}(t\vT)\quad{P\restr\s(\vT)-\mbox{a.s.}}
\]
holds true for each $t\in(0,\infty)$.
\smallskip

In particular, if the distribution of $\vT$ is degenerate at $\theta_0>0$ (i.e. $P_{\vT}(\{\theta_0\})=1$), then $\cnp$ is a {\bf $P$-Poisson process} with parameter $\theta_0$ (or a $P$-PP($\theta_0$) for short).
\smallskip

An \agpw $\agp$ is said to be a {\bf compound mixed Poisson process} on $(\vO,\vS,P)$ with parameters $\vT$ and $P_{X_1}$ 
(or a $P$-CMPP$(\vT,P_{X_1})$ for short), if it is induced by a risk process $(\cnp, \csp)$ such that $N$ is a $P$-MPP($\vT$).

In particular, if the distribution of $\vT$ is degenerate at $\theta_0>0$ then $\agp$ is said to be a {\bf compound Poisson process} on $(\vO,\vS,P)$ with parameters $\theta_0$ and $P_{X_1}$ (or a $P$-CPP$(\theta_0,P_{X_1})$ for short), if it is induced by a $P$-risk process $(N,X)$ 
such that $\cnp$ is a $P$-PP($\theta_0$). 
\smallskip

The following conditions are useful for the study of CMPPs: 
\begin{enumerate}
\item[\textnormal{(a1)}] 
The processes $\cip$ and $\csp$ are $P$-conditionally mutually independent.
\item[\textnormal{(a2)}] 
The random variable $\vT$ and the sequence $\csp$ are $P$-(unconditionally) independent.
\end{enumerate}
\medskip

{\em Next, whenever condition 
\textnormal{\textnormal{(a1)} and \textnormal{(a2)}} 
holds true we shall write that the quadruplet $(P,\cip,\csp,\vT)$ 
or (if no confusion arises) the probability measure $P$ satisfies
\textnormal{(a1)} and \textnormal{(a2)}, respectively}.
\smallskip

Consider now a second arbitrary but fixed probability space $(\vY,T,Q)$. The following definition is a special instance of that in \cite{fr4}, 452E, 
appropriate for our investigation.

\begin{df}\label{dis} 
\mg{dis}
A \textbf{regular conditional probability} (r.c.p. for short) \textbf{of $P$ over $Q$} is a family $\{P_{y}\}_{y\in\vY}$ of probability measures $P_y:\vS\longrightarrow\R$ such that 
\begin{enumerate}
\item[(d1)] for each $D\in\vS$ the function $P_{\cdot}(D):\vY\longrightarrow\R$ is $T$-measurable;
\item[(d2)] $\int P_{y}(D)Q(dy)=P(D)$ for each $D\in\vS$.
\end{enumerate}
We could instead use the term of {\em disintegration} instead, 
but it seems that it is better to reserve that term to the general case when
$P_y$'s may be defined on different domains (see \cite{pa}).

If $f:\vO\longrightarrow\vY$ is an inverse-measure-preserving map (i.e. $P(f^{-1}(B))=Q(B)$ for each $B\in{T}$), a r.c.p. $\{P_{y}\}_{y\in\vY}$ of $P$ over $Q$ is called \textbf{consistent} with $f$ if, for each $B\in{T}$, the equality $P_{y}(f^{-1}(B))=1$ 
holds for $Q$-almost all ($Q$-a.a. for short) $y\in B$.

We say that a r.c.p. $\{P_{y}\}_{y\in\vY}$ of $P$ over $Q$ consistent with $f$ 
is \textbf{essentially unique}, if for any other r.c.p. $\{P_{y}^{\prime}\}_{y\in\vY}$ of $P$ over $Q$ consistent with $f$ there exists a $Q$-null set $N\in{T}$ such that for any $y\notin N$ the equality $P_y=P_y^{\prime}$ holds true. 
\end{df}

\begin{rem}\label{magd} 
If $\vS$ is countably generated and 
$(\vO,\vS,P)$ or $P$ is perfect (see \cite{fa}, page 291 for the definition),
then there always exists a r.c.p. $\{P_{y}\}_{y\in\vY}$ of $P$ over $Q$ consistent with any inverse-measure-preserving map $f$ from $\vO$ into $\vY$ provided that $T$ is countably generated (see \cite{fa}, Theorems 6 and 3). Note that the most important applications in Probability Theory are still rooted in the case of standard Borel spaces $(\vO,\vS)$, that is, of spaces being isomorphic to $(Z,\mf{B}(Z))$, where $Z$ is some Polish space; hence of spaces satisfying always the above mentioned assumptions concerning $P$, $\vS$ and $T$. It is well-known that any Polish space is standard Borel; in particular, $\R^{d}$ and $\R^{\N}$ are such spaces. If $(\vO,\vS)$ and $(\vY,T)$ are non-empty standard Borel spaces, then there always exists an essentially unique r.c.p. $\{P_{y}\}_{y\in\vY}$ of $P$ over $Q$ consistent with any inverse-measure-preserving map $f$ from $\vO$ into $\vY$ (cf. e.g. \cite{fr4}, 452X(m)).
\end{rem}

{\em Throughout what follows we put $\vY:=(0,\infty)$ and assume that there exists a r.c.p. $\{P_{\theta}\}_{\theta\in\vY}$ of $P$ over $P_{\vT}$ 
consistent with $\vT$}.

If $\cnp$ is a $P$-MPP$(\vT)$,then the {\bf explosion} $E:=\{\sup_{n\in\N}T_n<\infty\}$ is 
a $P$-null set. In fact, by \cite{lm1}, Proposition 4.4 the \cnpw $\cnp$ is a $P_{\theta}$-PP($\theta$) for $P_{\vT}$-a.a. $\theta\in\vY$; hence 
$E$ is a $P_{\theta}$-null set for $P_{\vT}$-a.a. $\theta\in\vY$ by e.g. \cite{sch}, Corollary 2.1.5, implying that $E$ is a $P$-null set by condition 
(d2). Without loss of generality we may consider explosion equal to the empty set.

We need the following result of \cite{lm3} as 
a preparatory tool.

\begin{lem}\label{11}
\begin{enumerate}
\item
The following conditions are equivalent:
\begin{itemize}
\item[(a)] Condition \textnormal{(a1)};
\item[(b)] $\cnp$ and $\csp$ are $P$-conditionally mutually independent;
\item[(c)] 
there exists a $P_{\vT}$-null set $G^{\prime}\in\mf{B}(\vY)$ such that for any $\theta\notin{G}^{\prime}$ the processes $\cnp$ and $\csp$ are 
$P_{\theta}$-mutually independent.
\end{itemize}
\item
Condition \textnormal{(a2)} implies that the process $\csp$ is $P$-i.i.d. if and only it is $P$-conditionally i.i.d. if and only if there exists a $P_{\vT}$-null set $G^{\prime\prime}\in\mf{B}(\vY)$ such that 
for any $\theta\notin{G}^{\prime}$ the process $\csp$ is $P_{\theta}$-i.i.d..
\item
Conditions \textnormal{(a1)} and \textnormal{(a2)} imply that the pair $(\cnp,\csp)$ is a $P$-risk process if and only if
there exists a $P_{\vT}$-null set $G_*\in\mf{B}(\vY)$ such that for any $\theta\notin{G}_*$ the pair $\rp$ 
is a $P_{\theta}$-risk process.
\end{enumerate}
\end{lem}

The proof of the above result can be found 
in \cite{lm3}, Lemma 2.3. 


\section{Characterizations via martingales}\label{m:chm}

Let $\T\subseteq\R_+$ with $0\in\T$. For a process $Z_{\T}:=\{Z_t\}_{t\in\T}$ denote by $\mathcal{F}^Z_{\T}:=\{\mathcal{F}^Z_t\}_{t\in\T}$
the canonical filtration of $Z_{\T}$. For $\T=\R_+$ 
write $Z$ and $\mathcal{F}^Z$ in the place of $Z_{\R_+}$ and 
$\mathcal{F}^{{Z}}_{\R_+}$, respectively.
Write also $\mathcal{F}:=\{\mathcal{F}_t\}_{t\in\R_+}$, where $\mathcal{F}_t:=\sigma\bigl(\mathcal{F}^{\sagp}_t\cup\sigma(\vT)\bigr)$ for the canonical filtration of $\sagp$ and $\vT$, $\mathcal{F}_{\infty}^{\sagp}:=\sigma(\mathcal{F}^{\sagp})$ and 
$\mathcal{F}_{\infty}:=\s\bigl(\mathcal{F}^{\sagp}_{\infty}\cup\s(\vT)\bigr)$ for simplicity.

Recall that 
a {\bf martingale in} $\mathcal{L}^1(P)$ {\bf adapted to the filtration} 
$\mathcal{Z}_{\T}$, or else a $\mathcal{Z}_{\T}$-{\bf martingale} in $\mathcal{L}^1(P)$, is a process $Z_{\T}:=\{Z_t\}_{t\in\T}$ of  
real-valued random variables in $\mathcal{L}^1(P)$ such that $Z_t$ is $\mathcal{Z}_t$-measurable for each $t\in\T$ and whenever $s\leq{t}$ in $\T$ and $E\in\mathcal{Z}_s$, then $\int_EZ_sdP=\int_EZ_tdP$.
The latter condition is called the {\bf martingale property} (cf. e.g. \cite{sch}, page 25).
For  $\mathcal{Z}_{\R_+}=\mathcal{F}$ we simply say that $Z$ is a martingale in $\mathcal{L}^1(P)$.

\begin{rem}\label{qS1} 
\noindent
For any $n\in\N$ the random variable $X_n$ is $\mathcal{F}_{T_n}^{\sagp}$-measurable, where 
\[
\mathcal{F}_{T_n}^{\sagp}:=\{A\in\vS: A\cap\{T_n\leq{t}\}\in\mathcal{F}_t^{\sagp}\;\;\mbox{for every}\;\; t\in\R_+\},
\]
and for any $t\in\R_+$ the random variable $X_{N_t}$ is $\mathcal{F}_t^{\sagp}$-measurable.

In fact, it follows by \cite{sch}, Lemma 2.1.2 that all random variables $T_n$ are $\mathcal{F}^{\sagp}$-stopping times. 
Furthermore, $\agp$ is right-continuous since $\cnp$ is so. The latter together with the fact that $T_{n-1}<T_n$ for any $n\in\N$ yields that the random variables $S_{T_n}$ and $S_{T_{n-1}}$ are $\mathcal{F}_{T_n}^{\sagp}$-measurable for each $n\in\N$ (cf. e.g. \cite{ks}, Chapter 1, Propositions 2.18, 1.13 and Lemma 2.15). 
Thus, taking into account that $X_n=S_{T_n}-S_{T_{n-1}}$ since $N_{T_n}=n$ for each $n\in\N$, we deduce that $X_n$ is $\mathcal{F}_{T_n}^{\sagp}$-measurable for any $n\in\N$. 

But for all $n\in\N_0$ and $t\in\R_+$ we have $\{N_t=n\}=\{T_n\leq t<T_{n+1}\}\in\mathcal{F}_t^{\sagp}$ 
(see \cite{sch}, Lemma 2.1.2), implying that $X_{N_t}^{-1}(B)\cap\{N_t=n\}\in\mathcal{F}_t^{\sagp}$ 
for each $B\in\mf{B}(\vY)$ (see \cite{ks}, Chapter 1, Lemma 2.15). 
Consequently, the $\mathcal{F}_t^{\sagp}$-measurability of each random variable $X_{N_t}$ follows.
\end{rem}

Put $\mathcal{F}^{\cnp,\vT}:=
\{\mathcal{F}_t^{\cnp,\vT}\}_{t\in\R_+}$, where $\mathcal{F}_t^{\cnp,\vT}:=\s\bigl(\mathcal{F}_t^{\cnp}\cup\s(\vT)\bigr)$. Then
 $\mathcal{F}^{\cnp,\vT}$ is a filtration $(\vO,\vS)$. Moreover, set $\mathcal{F}_{\infty}^{\cnp}:=\s(\bigcup_{t\in\R_+}\mathcal{F}_t^{\cnp})$ and $\mathcal{F}_{\infty}^{\cnp,\vT}:=\s(\mathcal{F}_{\infty}^{\cnp}\cup\sigma(\vT))$.


\emph{Since our interest does not exceed the information generated by the \agpB, we assume throughout what follows that $\vS=\mathcal{F}_{\infty}$}.

\begin{rem}\label{16} 


Assume that the conditions $(\mathrm{a1})$ and $(\mathrm{a2})$ are satisfied by $(P,\csp,\cip,\vT)$. 
We then get that if the \cnpw has $P$-conditionally independent (and stationary) increments, then the same applies for the \agpB.


 


In fact, if $\cnp$ is a \cnpw with $P$-conditionally independent (and stationary) increments, then due to $(i)$ and $(ii)$ of \cite{lm1}, Lemma 4.2 and since $N_0=0$, this is equivalent to the fact that $\cnp$ has ${P_{\theta}}$-independent (and stationary) increments for 
$P_{\vT}$-a.a. $\theta\in\vY$. But since due to Lemma \ref{11}, $(iii)$ the pair $\rp$ is a risk process on $(\vO,\vS,P_{\theta})$, 
the $\agp$ having 
${P_{\theta}}$-independent (and stationary) increments for $P_{\vT}$-a.a. $\theta\in\vY$ is implied (cf. e.g. \cite{sch}, Theorem 5.1.2 
and 5.1.3). The latter, taking into account again $(i)$ and $(ii)$ of \cite{lm1}, Lemma 4.2, together with the fact that $S_0=0$, is in its own turn equivalent to the fact that $\agp$ has $P$-conditionally independent (and stationary) increments.

\end{rem}

Next we provide two lemmas which will turn to be useful for the proof of the main result of this section (see Theorem \ref{27}).

\begin{lem}\label{25}
Let $\agp$ be an \agpw induced by a risk process $(\cnp,\csp)$. If conditions \textnormal{(a1)} and \textnormal{(a2)} 
are satisfied by $(P,\csp,\cip,\vT)$ and the random variable $X_1$ is $P$-integrable, then 
for each $u,t\in\R_+$ with $u\leq t$ and for each $A\in\mathcal{F}_u^{\cnp,\vT}$ 
the equality
\[
\int_AS_tdP_{\theta}=\int_AN_t\E_{P_{\theta}}[X_1]dP_{\theta}
\] 
holds true 
for $P_{\vT}$-a.a. $\theta\in\vY$.
\end{lem}

\begin{proof}
First notice that $(\mathrm{a1})$ together with Lemma \ref{11}, $(i)$, imply that $\cnp$ and $\csp$ are $P_{\theta}$-independent for 
$P_{\vT}$-a.a. $\theta\in\vY$. Then, taking also into account the $P$ - integrabillity of $X_1$, condition (a2) and Lemma \ref{11}, $(iii)$, 
for each $u,t\in\R_+$ with $u\leq t$ and for each $A\in\mathcal{F}_u^{\cnp,\vT}$ and for $P_{\vT}$-a.a. $\theta\in\vY$ we have
\begin{eqnarray*}
\int_A S_tdP_{\theta} &=& \int_A\sum_{n=0}^{\infty}\chi_{\{N_t=n\}}\sum_{k=1}^nX_kdP_{\theta}
=\sum_{n=0}^{\infty}\int_A\chi_{\{N_t=n\}}\sum_{k=1}^nX_kdP_{\theta}\\
&=& \sum_{n=0}^{\infty}\int\sum_{k=1}^n\chi_{A\cap\{N_t=n\}}X_kdP_{\theta}	
=\sum_{n=0}^{\infty}\sum_{k=1}^n\E_{P_{\theta}}[\chi_{A\cap\{N_t=n\}}X_k]\\
&=&
\sum_{n=0}^{\infty}\sum_{k=1}^n\E_{P_{\theta}}[\chi_{A\cap\{N_t=n\}}]\E_{P_{\theta}}[X_k]
=\sum_{n=0}^{\infty}n\E_{P_{\theta}}[\chi_{A\cap\{N_t=n\}}]\E_{P_{\theta}}[X_1]\\
&=& \sum_{n=0}^{\infty}n\int_A\chi_{\{N_t=n\}}\E_{P_{\theta}}[X_1]dP_{\theta}
=\int_A\sum_{n=0}^{\infty}n\chi_{\{N_t=n\}}\E_{P_{\theta}}[X_1]dP_{\theta}\\
&=& \int_AN_t\E_{P_{\theta}}[X_1]dP_{\theta},
\end{eqnarray*}
since $N_t=\sum_{n=0}^{\infty}n\chi_{\{N_t=n\}}$ for each $t\in\R_+$.
\end{proof}

\begin{lem}\label{26}
Let $\agp$ be an \agpw induced by a risk process $(\cnp,\csp)$ and assume that 
conditions $(\mathrm{a1})$ and $(\mathrm{a2})$ are satisfied by $(P,\csp,\cip,\vT)$, and that 
the random variables $X_1$ and $\vT$ are $P$-integrable. 

\noindent If the process $\{S_t-t\vT\E_P[X_1]\}_{t\in\R_+}$ is a martingale in $\mathcal{L}^1(P)$, then the process $\{N_t-t\vT\}_{t\in\R_+}$ is a $\mathcal{F}^{\cnp,\vT}$-martingale in $\mathcal{L}^1(P)$.
\end{lem}

\begin{proof}
Clearly the process $\{N_t-t\vT\}_{t\in\R_+}$ is adapted to the filtration $\mathcal{F}^{\cnp,\vT}$. Assume now that $\{S_t-t\vT\E_P[X_1]\}_{t\in\R_+}$ is a martingale in $\mathcal{L}^1(P)$.

\begin{it}
Claim 1. For each $t\in\R_+$ the random variable $N_t-t\vT$ is $P$-integrable.
\end{it}

{\em Proof.} 
It follows by the martingale property of $\{S_t-t\vT\E_P[X_1]\}_{t\in\R_+}$ that for any $u,t\in\R_+$ with $u\leq{t}$ we have 
\[
\E_P\bigl[S_t-t\vT\E_P[X_1]\bigr]=\E_P\bigl[S_u-u\vT\E_P[X_1]\bigr] 
\]
implying for $u=0$ that 
\begin{equation}\label{clm1}
\E_P[S_t]=t\E_P[\vT]\E_P[X_1].
\end{equation} 
Fix on arbitrary $t\in\R_+$. We then get 
\[
\E_P[N_t]\E_P[X_1]=t\E_P[\vT]\E_P[X_1];
\]
hence taking into account the $P$-integrability of $X_1$, we obtain that $\E_P[N_t]=t\E_P[\vT]$, implying that $N_t$ is $P$-integrable, since $\vT$ is so. Thus, $N_t-t\vT\in\mathcal{L}^1(P)$.$\hfill\Box$ 



\begin{it}
Claim 2. The process $\{N_t-t\vT\}_{t\in\R_+}$ satisfies the martingale property.
\end{it}

{\em Proof.} First
fix on arbitrary $u,t\in\R_+$ with $u\leq t$. Since $\{S_t-t\vT\E_P[X_1]\}_{t\in\R_+}$ is a martingale in $\mathcal{L}^1(P)$, applying \cite{lm1}, Lemma 4.6 together with \cite{lm1err}, we get that for $P_{\vT}$-a.a. $\theta\in\vY$ 
the process $\{S_t-t\theta\E_{P_{\theta}}[X_1]\}_{t\in\R_+}$ is a martingale in $\mathcal{L}^1(P_{\theta})$; hence
\[
\int_B(S_t-t\theta\E_{P_{\theta}}[X_1])dP_{\theta}=\int_B(S_u-u\theta\E_{P_{\theta}}[X_1])dP_{\theta}
\quad\mbox{for each}\quad B\in\mathcal{F}_u.
\]
Let us fix now on arbitrary $A\in\mathcal{F}_u^{\cnp,\vT}$. Then the last condition along with the inclusion $\mathcal{F}_u^{\cnp,\vT}\sq\mathcal{F}_u$ implies 
that
\[
\int_A(S_t-t\theta\E_{P_{\theta}}[X_1])dP_{\theta}=\int_A(S_u-u\theta\E_{P_{\theta}}[X_1])dP_{\theta}
\quad\mbox{for $P_{\vT}$-a.a. $\theta\in\vY$}.
\]
The 
latter together with Lemma \ref{25} equivalently 
yields that
\[
\int_A(N_t-t\theta)\E_{P_{\theta}}[X_1]dP_{\theta}
=\int_A(N_u-u\theta)\E_{P_{\theta}}[X_1]dP_{\theta}\;\mbox{for $P_{\vT}$-a.a. $\theta\in\vY$}, 
\]
equivalently
\[
\int_A(N_t-t\theta)dP_{\theta}
=\int_A(N_u-u\theta)dP_{\theta}\;\;\mbox{for}\;\; P_{\vT}-\mbox{a.a.}\; \theta\in\vY, 
\]
since $\E_{P_{\theta}}[X_1]=\E_P[X_1]$ for $P_{\vT}$-a.a. $\theta\in\vY$ by \cite{lm3}, first equality of condition (4.1) in 
step (b) of the proof of Proposition 4.4. Consequently, the process $\{N_t-t\theta\}_{t\in\R_+}$ satisfies the martingale property. This completes the proof of Claim 2 as well as the whole proof.
\end{proof}

\begin{thm}\label{17} 
Let $\agp$ be an \agpw induced by a risk process $(\cnp,\csp)$ and assume that 
conditions $(\mathrm{a1})$ and $(\mathrm{a2})$ are satisfied by $(P,\csp,\cip,\vT)$, and that 
the random variables $X_1$ and $\vT$ are $P$-integrable. 

\noindent Then the following are equivalent:
\begin{enumerate}
\item 
The \agpw $\agp$ is a $P$-CMPP($\vT,P_{X_1}$);
\item
the process $\{S_t-t\vT\E_P[X_1]\}_{t\in\R_+}$ is a martingale in $\mathcal{L}^1(P)$.
\end{enumerate}
\end{thm}

\begin{proof}
Ad $(i)\Longrightarrow(ii)$:
If $(i)$ holds, then $\cnp$ is a $P$-MPP($\vT$). Moreover, 
the $P$-integrability of $X_1$ and $\vT$ yields
by Wald's identities (cf. e.g. \cite{sch}, Lemma 5.2.10) that $\E_P[S_t]=t\E_{P}[\vT]\E_P[X_1]<\infty$ for each $t\in\R_+$.

Since $\cnp$ is a $P$-MPP($\vT$), the \agpw $\agp$ has $P$-conditionally independent increments by Remark \ref{16}. Then \cite{lm1}, Proposition 4.8 together with \cite{lm1err} yields that the process $\{S_t-\E_{P_{\theta}}[S_t]\}_{t\in\R_+}$ is a martingale in $\mathcal{L}^1(P_{\theta})$ for 
$P_{\vT}$-a.a. $\theta\in\vY$, implying along with the $P$-integrability of the random variables $S_t$, $t\in\R_+$, that 
$\{S_t-\E_{P}[S_t\mid\vT]\}_{t\in\R_+}$ is a martingale in $\mathcal{L}^1(P)$ (see \cite{lm1}, Lemma 4.6 together with \cite{lm1err}); 
hence $(ii)$ follows.

Ad $(ii)\Longrightarrow(i)$: If $(ii)$ holds, then according to Lemma \ref{26}, the process 
$\{N_t-t\vT\}_{t\in\R_+}$ is a $\mathcal{F}^{\cnp,\vT}$-martingale in $\mathcal{L}^1(P)$.
It then follows by the martingale 
property of $\{N_t-t\vT\}_{t\in\R_+}$ that for each $t\in\R_+$ the equality $\E_P[N_t\mid\vT]=t\vT$ holds $P\restr\s(\vT)$-a.s., implying
that $\{N_t-\E_P[N_t\mid\vT]\}_{t\in\R_+}$ is a is a $\mathcal{F}^{\cnp,\vT}$-martingale in $\mathcal{L}^1(P)$; 
hence we can apply \cite{lm1}, Lemma 4.6, in order to deduce that the process $\{N_t-t\theta\}_{t\in\R_+}$ is a 
$\mathcal{F}^{\cnp,\vT}$-martingale in $\mathcal{L}^1(P_{\theta})$ 
for $P_{\vT}$-a.a. $\theta\in{\vY}$. 
The latter together with e.g. \cite{sch}, Theorem 2.3.4, yields that $\cnp$ is a $P_{\theta}$-PP($\theta$); hence it is a 
$P$-MPP($\vT$) by \cite{lm1}, Proposition 4.4, implying that $\agp$ is a $P$-CMPP($\vT,P_{X_1}$).
This completes the proof.
\end{proof}

\smallskip

Summarizing up we obtain:

\begin{thm}\label{27}
Let $\agp$ be an \agpw induced by a risk process $(\cnp,\csp)$ and assume that 
conditions $(\mathrm{a1})$ and $(\mathrm{a2})$ are satisfied by $(P,\csp,\cip,\vT)$, and that 
the random variables $X_1$ and $\vT$ are $P$-integrable. 
\noindent The following statements are equivalent:
\begin{enumerate}
\item 
The \agpw $\agp$ is a $P$-CMPP($\vT,P_{X_1}$);
\item 
the process $\{S_t-t\vT\E_{P}[X_1]\}_{t\in\R_+}$ is a martingale in $\mathcal{L}^1(P)$;
\item
the process $\{S_t-t\theta\E_{P_{\theta}}[X_1]\}_{t\in\R_+}$ is a 
martingale in $\mathcal{L}^1(P_{\theta})$ 
for $P_{\vT}$-a.a. $\theta\in\vY$.
\item 
the process $\{N_t-t\vT\}_{t\in\R_+}$ is a $\mathcal{F}^{N,\vT}$-martingale 
in $\mathcal{L}^1(P)$;
\item 
the process $\{N_t-t\theta\}_{t\in\R_+}$ is a 
$\mathcal{F}^{\cnp,\vT}$-martingale in $\mathcal{L}^1(P_{\theta})$ 
for $P_{\vT}$-a.a. $\theta\in\vY$.
\end{enumerate}
\end{thm}

\begin{proof}

The equivalences $(i)\Longleftrightarrow(ii)$ and $(ii)\Longleftrightarrow(iii)$
follow by Theorem \ref{17} and \cite{lm3}, Proposition 6.2, respectively. The implication $(ii)\Longrightarrow (iv)$ 
is a consequence of Lemma \ref{26}, while the implication $(iv)\Longrightarrow(v)$ follows by \cite{lm1}, Lemma 4.6 along  with \cite{lm1err}. Finally, according to e.g. \cite{sch}, Theorem 2.3.4, the 
statement $(v)$ is equivalent to the fact that $\cnp$ is a $P_{\theta}$-PP($\theta$) for $P_{\vT}$-a.a. $\theta\in\arch{\vY}$, 
implying that $\cnp$ is a $P$-MPP($\vT$) by \cite{lm1}, Proposition 4.4; hence $\agp$ is a $P$-CMPP($\vT,P_{X_1}$). This completes the proof. 
\end{proof}

The following consequence of Theorem \ref{27} contains the S. Watanabe's martingale characterization of the Poisson 
process with parameter $\theta_0>0$ (see Watanabe \cite{wa}). 

\begin{cor}\label{28}
Assume that the \agpw $\agp$ is induced by the risk process $(\cnp,\csp)$ and that the random variable $X_1$ is $P$-integrable.
\noindent For $\theta_0>0$ the following 
statements are equivalent:
\begin{enumerate}
\item
The \agpw process $\agp$ is a $P$-CPP($\theta_0,P_{X_1}$);
\item
the process $\{S_t-t\theta_0\E_P[X_1]\}_{t\in\R_+}$ is a $\mathcal{F}^{\agp}$-martingale in $\mathcal{L}^1(P)$;
\item 
the process $\{N_t-t\theta_0\}_{t\in\R_+}$ is a $\mathcal{F}^{\cnp}$-martingale in $\mathcal{L}^1(P)$;
\item
the \cnpw  $\cnp$ is a $P$-PP($\theta_0$).
\end{enumerate}
\end{cor}

\begin{proof}
Let $\vT$ be a $P$-a.s. positive random variable on $\vO$ such that $P_{\vT}=\delta_{\theta_0}$. Then $(P,\csp,\cip,\vT)$ satisfies  conditions $(\mathrm{a1})$ and $(\mathrm{a2})$, $P=P_{\theta_0}$ by 
property $(\mathrm{d2})$, and statement $(i)$ is equivalent to statement
\begin{equation}\label{28a}
\mbox{the aggregate process $\agp$ is a $P$-CMPP($\vT,P_{X_1}$)}.
\end{equation} 
Thus, applying Theorem \ref{27}, we obtain that $(i)$ is equivalent to each of its statements $(iii)$ and $(v)$. But statement $(iii)$ of 
Theorem \ref{27} says that there exists a $P_{\vT}$-null set $L\in\mf{B}(\vY)$ such that the process 
$\{S_t-\theta{t}\E_{P_{\theta}}[X_1]\}_{t\in\R_+}$ is an $\mathcal{F}^{\agp}$-martingale in $\mathcal{L}^1(P)$ for any $\theta\notin{L}$. 
Because 
$P(\{\vT\not=\theta_0\})=0$, we get that $L=\vY\setminus\{\theta_0\}$ and $\theta_0\notin{L}$; hence statement $(ii)$ is equivalent to statement $(iii)$ of Theorem \ref{27}. In the same way, statement $(iii)$ is equivalent to the statement $(v)$ of Theorem \ref{27}. 
The equivalence $(i)\Longleftrightarrow(iv)$ follows by the definitions. This completes the proof.
\end{proof}

\begin{rem}\label{30}
An immediate consequence of Corollary \ref{28} is Proposition 4.1, $(i)\Longleftrightarrow(ii)$, of \cite{mt3}, according to which if $\agp$ is a compound renewal process with parameters $\mathbf{K}(\theta)$ and $P_{X_1}$,  where $\theta:=\frac{1}{\E_P[X_1]}$, such that $X_1$ and $W_1$ are $P$-integrable (see \cite{mt3} for the definitions), then the processes $\{Z_t\}_{t\in\R_+}$ with 
\[Z_t:=\agp_t-t\frac{\E_P[X_1]}{\E_P[W_1]}\quad\mbox{for all}\quad t\in\R_+
\]
is a $\mathcal{F}^{\agp}$-martingale in $\mathcal{L}^1(P)$ if and only if the \cnpw $\cnp$ is a $P$-PP($\theta$). The proof of Proposition 4.1, $(i)\Longleftrightarrow(ii)$ of \cite{mt3} is totally different than ours. 
\end{rem}

\begin{rem}\label{29}
If, for $P_{\vT}$-a.a. $\theta\in\vY$, the distribution of $X_1$ under $P_{\theta}$ is degenerate at 1, then Theorem \ref{27} reduces 
to the main result of \cite{lm1} (see \cite{lm1}, Theorem 4.10, assertions $(ii)$ to $(vi)$ along with \cite{lm1err}).

In fact, if $(P_{\theta})_{X_1}=\delta_1$ for $P_{\vT}$-a.a. $\theta\in\vY$, then $P_{X_1}=\delta_1$ and there exists a $P$-null and $P_{\theta}$-null set $\vO_X\in\mathcal{F}$ such that $X_n(\omega)=1$ for all $\omega\in\vO\setminus\vO_X$, implying that there exists a $P$-null and $P_{\theta}$-null set $\vO_{\infty}\in\mathcal{F}$ such that 
\[
\sum_{k=1}^{\infty}X_k=\infty
\]
for all $\omega\in\vO\setminus\vO_{\infty}$. Define $\vO_S:=\vO_X\cup\vO_{\infty}$. It then follows tat $S$ is a $P$- and $P_{\theta}$-\cnpw with exceptional null set $\vO_S$ for $P_{\vT}$-a.a. $\theta\in\vY$. Therefore, Theorem \ref{27} reduces to the main result of \cite{lm1}.  
\end{rem}

\bigskip

\bigskip

\bigskip

\bigskip

{\small
\begin{tabular}{lr}
{\tt D.P. Lyberopoulos} & {\tt N.D. Macheras}\\
{\rm Retail Price Indices Section} & {\rm Dept. of Statistics and Insurance Science}\\
{\rm Hellenic Statistical Authority (ELSTAT), Greece}& {\rm University of Piraeus, Greece}\\
{\rm E-mail:} {\tt d.lymperopoulos@statistics.gr} & {\rm E-mail:} {\tt macheras@unipi.gr}\\
\qquad\quad\;\;{\tt dilyber@webmail.unipi.gr}&
\end{tabular}
}

\end{document}